\documentclass[12 pt,onecolumn]{IEEEtran}
\pdfoutput=1 

\usepackage{amssymb,epsfig,color,cite,amsmath,amsfonts,mathrsfs,algorithm,algorithmicx}
\usepackage{epstopdf}
\usepackage{datetime,fancyhdr}
\usepackage{algpseudocode}
\usepackage{arydshln}
\usepackage{multirow}
\usepackage{float}
\usepackage{amsthm} 
\usepackage{xcolor}

\usepackage{times} 
\usepackage{epsfig}
\usepackage{graphicx}
\usepackage{latexsym}
\usepackage{subfigure}

\newtheorem{lemma}{Lemma}[section]
\newtheorem{theorem}{Theorem}[section]

\newtheorem{remark}{Remark}[section]

\newtheorem{proposition}{Proposition}[section]
\newtheorem{assumption}{Assumption}[section]

\newcommand{\upcite}[1]{\textsuperscript{\textsuperscript{\cite{#1}}}}

\ifCLASSINFOpdf
\else
\fi

\begin{document}

\title{Distributed proximal gradient algorithm for non-smooth non-convex optimization over time-varying networks
}

\author{Xia Jiang,
	Xianlin~Zeng,~\IEEEmembership{Member,~IEEE,}
	Jian~Sun,~\IEEEmembership{Member,~IEEE,}
	and~Jie~Chen,~\IEEEmembership{Fellow,~IEEE,}
	\thanks{This work was supported in part by the National Natural Science Foundation of China under Grants 61925303, 62088101, 62073035 and the National Key Research and Development Program of China under Grant 2018YFB1700100. \emph{(Corresponding author: Jian Sun.)}}
	\thanks{X. Jiang (jiang-xia@bit.edu.cn) and J. Sun (sunjian@bit.edu.cn) are with Key Laboratory of Intelligent Control and Decision of Complex Systems, School of Automation, Beijing Institute of Technology, Beijing, 100081, China, and also with the Beijing Institute of Technology Chongqing Innovation Center, Chongqing  401120, China}
	\thanks{X. Zeng (xianlin.zeng@bit.edu.cn) is with Key Laboratory of Intelligent Control and Decision of Complex Systems, School of Automation, Beijing Institute of Technology, Beijing, 100081, China}
	\thanks{J. Chen (chenjie@bit.edu.cn) is with Beijing Advanced Innovation Center for Intelligent Robots and Systems (Beijing Institute of Technology), Key Laboratory of Biomimetic Robots and Systems (Beijing Institute of Technology), Ministry of Education, Beijing, 100081, China, and also with the School of Electronic and Information Engineering, Tongji University, Shanghai, 200082, China}}

\maketitle

\begin{abstract}
This note studies the distributed non-convex optimization problem with non-smooth regularization, which has wide applications in decentralized learning, estimation and control. The objective function is the sum of different local objective functions, which consist of differentiable (possibly non-convex) cost functions and non-smooth convex functions. This paper presents a distributed proximal gradient algorithm for the non-smooth non-convex optimization problem over time-varying multi-agent networks. Each agent updates local variable estimate by the multi-step consensus operator and the proximal operator. We prove that the generated local variables achieve consensus and converge to the set of critical points with convergence rate $O(1/T)$. Finally, we verify the efficacy of proposed algorithm by numerical simulations. 
\end{abstract}


\section{Introduction}
\par Motivated by many  problems in signal processing and machine learning over  networks, distributed non-smooth non-convex optimization has attracted significant attention.  In this problem setup, each node in the network only knows local function information and communicates with its neighbors to solve the global optimization problem. One fundamental model for distributed non-smooth non-convex optimization, arising  from optimization problems such as Lasso\upcite{lasso_nonsmooth}, SVM\upcite{svm_nonsmooth}, and optimizing neural networks\upcite{infor_nonsmooth}, is that each local objective function of a node is the summation of a (non-convex) differentiable function and a non-smooth convex function ($l_1$ norm or indicator function). Although the research on distributed optimization has made significant progress on non-smooth convex problems\upcite{nonsmooth_tac,nonsmooth_convex_ijrnc,disnonsmooth_tac,svm_dis,zeng_tac}, distributed non-smooth non-convex optimization is still challenging.
\par Researchers have made great achievements in centralized and parallel algorithms for non-smooth non-convex optimization problems\upcite{nonsmooth_nonconvex_overview,cen-PALM,Yin_2017,GuWHH18,inexact_lei,sw,JMLRPGA}. For instance, \cite{cen-PALM} developed a proximal alternating linearized minimization algorithm with global convergence under Kurdyka-Lojasiewicz property. \cite{Yin_2017} extended the two blocks of objective function in \cite{cen-PALM} to multiple blocks and introduced extrapolation to accelerate the block prox-linear method. When the proximal operator does not have an analytic solution or exactly solving the proximal operator is time-consuming, \cite{GuWHH18,inexact_lei} studied some inexact proximal gradient algorithms for non-convex optimization. With the explosion of data and the development of distributed network systems, \cite{sw,JMLRPGA} developed some asynchronous parallel methods with considerations of unreliable communication links. However, with privacy or security considerations, it is necessary to design fully distributed algorithms for large-scale non-smooth non-convex optimization.
\par In recent years, some distributed discrete-time algorithms\upcite{decen-nonconvex,hong_nonconv,scutari_timevarying,NEXT2016,constrianed_prox} have been proposed for non-smooth non-convex optimization over multi-agent networks. Over time-invariant graphs, \cite{decen-nonconvex,hong_nonconv} proposed distributed proximal gradient algorithms for (non-smooth) non-convex optimization with convergence to consensus stationary solutions. However, time-invariant graphs are difficult and expensive to hold for practical multi-agent networks. Over time-varying networks, \cite{NEXT2016} developed a distributed discrete-time algorithm with successive convex approximation and dynamic consensus mechanism. If agents only have noisy observations of local functions, \cite{constrianed_prox} proposed a distributed stochastic approximation algorithm over time-varing graphs without requiring objective functions be convex and Lipschitz continuous. However, the diminishing step-sizes in existing algorithms hinder the convergence performance. This paper studies a distributed algorithm with a constant step-size for non-smooth non-convex optimization over time-varying communication graphs.
\par The contributions of this paper are summarized as follows.
\begin{itemize}
	\item The paper proposes one distributed proximal gradient algorithm for non-smooth non-convex optimization over time-varying multi-agent networks. The proposed algorithm adopts the multi-step consensus stage to make local variable estimates closer to each other and extends the recent distributed proximal algorithm \cite{decen-nonconvex} over time-invariant graphs to time-varying network graphs. What's more, the proposed algorithm owns a constant  step-size, overcoming the shortage of diminishing step-sizes that hinder the convergence performance\cite{NEXT2016,constrianed_prox}.
	\item We provide complete and rigorous convergence proofs for the proposed distributed proximal gradient algorithm. The proposed algorithm over time-varying graphs has a same convergence rate $O(\frac{1}{T})$ as the algorithm over time-invariant graphs \cite{decen-nonconvex}. To the best of our knowledge, for non-smooth non-convex optimization problems over time-varying graphs, this is the first convergence result showing the rate of convergence of distributed algorithms without using successive convex approximation. 
\end{itemize}
\par The remainder of the paper is organized as follows. 
The preliminary mathematical notations, graph theory and proximal operator are introduced in section \ref{math_sec}. The optimization problem description and the design of a distributed solver are provided in section \ref{solver_design}. The convergence performance of the proposed algorithm is proved theoretically in section \ref{proof_sec}. The numerical simulations are provided in section \ref{simulation} and the conclusion is made in section \ref{conclusion}.

\section{ Preliminaries } \label{math_sec}
\subsection{Mathematical notations \& graph theory}
\par We write $\mathbb{R}$ as the set of real numbers, $\mathbb{N}$ as the set of natural numbers, $\mathbb{R}^n$ as the set of $n$-dimensional real column vectors and $\mathbb{R}^{n\times m}$ as the set of $n$-by-$m$ real matrices, respectively. We denote $v'$ as the transpose of a vector $v$. In addition, $\left\|\cdot \right\|$ denotes the Euclidean norm, $\langle\cdot,\cdot\rangle$ denotes the inner product, which is defined by $\langle a,b\rangle=a'b$ and $\lceil a \rceil$ deontes the smallest integer greater than real number $a$. The vectors in this paper are column vectors unless otherwise stated. For a differentiable function $g:\mathbb{R}^n \to \mathbb{R}$, $\nabla g(x)$ denotes the gradient of function $g$ with respect to $x$. The $\varepsilon$-subdifferential of a convex function $h$ at $x$ is the set of vectors $y$ such that $h(z)-h(x)\geq y^{T}(z-x)-\varepsilon$ for all $z$.
\par The dynamic communication among $m$ agents over time-varying undirected topology is often modeled as $\mathcal{G}(\mathcal{V},\mathcal{E}(t),{A}(t))$, where $\mathcal{V}=\{1,\dots,m\}$ is a finite nonempty node set with $i$ representing $i$th node, $\mathcal{E}(t)\subset \mathcal{V} \times \mathcal{V}$ is the time-varying edge set. The adjacent matrix is denoted by ${A}(t)=[a_{ij}(t)]$ $\in$ $\mathbb{R}^{m \times m}$ such that $a_{ij}(t)=a_{ji}(t) > 0$ if $\{i,j\} \in \mathcal{E}(t)$ and the elements $a_{ij}(t)=0$ otherwise. Note that adjacent matrices of undirected graphs are symmetric matrices. If an edge $\{i,j\} \in \mathcal{E}(t)$, then node $j$ is called a neighbor of node $i$.
\subsection{Proximal Operator}
\par For a proper non-differentiable convex function $h:\mathbb{R}^n\to (-\infty,\infty]$ and a scalar $\alpha>0$, the proximal operator is defined as
\begin{align}\label{prox_ope}
{\rm prox}_{\alpha,h}(x)={\rm argmin}_{z\in \mathbb{R}^n}{h(z)+\frac{1}{2\alpha}\|z-x\|^2}.
\end{align} 
\par The minimum is attained at a unique point $y= {\rm prox}_{\alpha,h}(x)$, which means the proximal operator is a single-valued map. In addition, it follows from the optimality condition for convex optimization problems that
\begin{align}
0\in \partial h(y)+\frac{1}{\alpha}(y-x),
\end{align} 
where the set $\partial h(y)$ is the subdifferential of non-differentiable function $h$ at $y$. The following proposition presents some properties of the proximal operator.
\begin{proposition}\label{prox_proposition}\cite{prox_pro}
	Let $h:\mathbb{R}^n\to (-\infty,\infty]$ be a closed proper convex function. For a scalar $\alpha>0$ and $x\in \mathbb{R}^n$, let $y={\rm prox}_{\alpha,h}(x)$.
	\begin{itemize}
		\item[(a)] The relationship $h(u)\geq h(y)+\frac{1}{\alpha}\langle x-y,u-y\rangle$ holds for all $u\in \mathbb{R}^n$.
		\item[(b)] For $x,\hat{x}\in \mathbb{R}^n$, 
		$$\|{\rm prox}_{\alpha,h}(x)-{\rm prox}_{\alpha,h}(\hat{x})\|\leq \|x-\hat{x}\|.$$
		\item[(c)] The vector $y$ can be written as $y=x-\alpha z$, where $z\in\partial h(y)$.
		\item[(d)] We have $\frac{1}{\alpha}(x-y)\in \partial h(y)$.
	\end{itemize}
\end{proposition}
\par When there exist errors in the computation of proximal operators, we denote the inexact proximal operator by ${\rm prox}_{\alpha,h}^{\varepsilon}(\cdot)$. Let $x_k$ denote the variable at iteration $k$ and $\varepsilon_k$ denote the error in the proximal objective function. Then, the inexact proximal operator at iteration $k$ is
\begin{align}\label{inprox_def}
x_k\in {\rm prox}_{\alpha,h}^{\varepsilon_k}(y)\triangleq \Big\{&\tilde{x}\big|\frac{1}{2\alpha}\|\tilde{x}-y\|^2+h(\tilde{x})\leq \varepsilon_k+\min_{x\in \mathbb{R}^n}\big\{\frac{1}{2\alpha}\|x-y\|^2+h(x)\big\}\Big\}.
\end{align}
\section{Problem Description and Distributed Solver}\label{solver_design}
Consider a multi-agent system composed of $m$ agents, which are interconnected by a time-varying communication network. We aim to design a distributed algorithm for the multi-agent system to solve the following optimization problem
\begin{align}\label{opti_pro}
	{\rm min}_{x} \ &f(x)=\frac{1}{m}\sum_{i=1}^m\big(g_i(x)+h(x)\big),
\end{align}
where $x\in \mathbb{R}^n$ is the decision variable, function $g_i$ is a differentiable (possibly non-convex) local cost function, and function $h$ is a non-smooth and convex regular function. For each agent $i$ in the network, $x_i\in \mathbb{R}^{n}$ is the local estimate of variable $x$. 
\begin{remark}
	The non-convexity of function $g_i$ makes it difficult to design an efficient convergent algorithm with rigorous proofs of optimality. In this paper, we prove that the proposed algorithm converges to the set of critical points. Although there exist some centralized works, it is not straightforward to extend them to distributed cases since the convergence may not hold with the influence of distributed nature.  
\end{remark}
\par Through this paper, we assume that the following standard assumptions hold for the optimization problem \eqref{opti_pro}.
\begin{assumption}\label{pro_assum}
	\begin{itemize}
		\item[(a)] For each agent $i$, $g_i$ is continuously differentiable and has a Lipschitz-continuous gradient with constant $L>0$,
		\begin{align}\label{l_assump}
		\left\|\nabla g_i(x)-\nabla g_i(y)\right\|\leq L\left\|x-y\right\|,
		\end{align}
		which implies that 
		\begin{align}\label{lfsmooth}
		g_i(x)\leq g_i(y)+\left<\nabla g_i(y),x-y\right>+\frac{L}{2}\left\|x-y\right\|^2.
		\end{align}
		\item[(b)] The regular function $h$ is convex.
		\item[(c)] There exists a scalar $G_g$ such that for each agent $i$, $\|\nabla g_i(x)\|<G_g$.
		\item[(d)] There exists a scalar $G_h$ such that for each sub-gradient $z\in\partial h(x)$, $\|z\|<G_h$.
		\item[(e)] The optimization problem owns at least one optimal solution $x^*$.
	\end{itemize}
\end{assumption}
\par Then, we propose the following distributed proximal gradient algorithm for solving (\ref{opti_pro}). For $i\in\{1,\ldots, m\}$, $k=1,\cdots$,
\begin{subequations}\label{distri_m}
	\begin{align}
	&q_{i,k+1}=x_{i,k}-\alpha \nabla g_i(x_{i,k}),\label{q_update}\\
	&v_{i,k+1}=\sum_{j=1}^m \lambda_{ij,k+1} q_{j,k+1}, \label{v_update}\\
	&x_{i,k+1}=prox_{\alpha,h}(v_{i,k+1}) \label{x_update},
	\end{align}
\end{subequations}
where $\alpha<\frac{1}{L}$ is a constant step-size which is also used in the proximal operator, $\lambda_{ij,k}$ is the $(i,j)$th element of matrix $\Phi\big(t(k)+k,t(k)\big)$, 
$$\lambda_{ij,k}=\big[\Phi\big(t(k)+k,t(k)\big)\big]_{ij},$$
where $t(k)$ is the total number of communication steps before iteration $k$ and $\Phi$ is a transition matrix, which is defined as
$$\Phi(t,s)=A(t)A(t-1)\cdots A(s+1) A(s),\ t>s\geq 0,$$
where $A(t)$ is the adjacent matrix at time $t$.

\par Before analyzing the behavior of proposed algorithm \eqref{distri_m} over a time-varying network, we introduce the following assumption.
\begin{assumption}\label{net_assum}
	Consider the undirected time-varying network with adjacent matrices $A(t)=[a_{ij}(t)]$, $t=1,2,\cdots$
	\begin{itemize}
		\item [(a)] For each $t$, the adjacent matrix $A(t)$ is doubly stochastic.
		\item [(b)] There exists a scalar $\eta\in (0,1)$ such that $a_{ii}(t)\geq \eta$ for all $i\in \{1,\cdots,m\}$. In addition, $a_{ij(t)}\geq \eta$ if if $\{i,j\}\in \mathcal E(t)$ and $a_{ij}(t)=0$ otherwise.
		\item [(c)] The time-varying graph sequence $\mathcal{G}_k$ is uniformly connected, which means that agent $j$ receives information from $i$ for infinitely many $t$. Moreover, there exists an integer $B\geq 1$ such that agent $i$ sends its information to all other agents at least once every $B$ consecutive time slots.
	\end{itemize}
\end{assumption} 
\begin{remark}
	In this assumption, part (a) guarantees that the variable estimates of neighbors impose an equal influence on the local variable estimate. Part (b) means that each agent gives significant weight to its current estimate and the estimates received from its neighbors. Part (c) states that the time-varying network is capable of exchanging information between any pair of agents in bounded time.
\end{remark}
\begin{remark}
	{Over time-varying graphs,} the proposed updating (\ref{v_update}) represents that agents perform $k$ rounds of communication steps at iteration $k$, which may be expensive as iteration number $k$ increases. However, when the time-varying graph is periodic, the updating (\ref{v_update}) is easy to compute due to the fact that $\lambda_{ij,k}$ is also periodic, which has been investigated in \cite{Nash_lou}. What's more, if the bounded intercommunication interval $B$ is known, the number of communication steps taken at iteration $k$ is significant reduced and the convergence performance is further improved, which has been discussed in \cite{proximal-convex}.   
\end{remark}

\section{Main Result}\label{proof_sec}
In this section, we present theoretical proofs for the convergence properties of proposed distributed algorithm. Let $\bar{x}_k\triangleq \frac{1}{m} \sum_{i=1}^m x_{i,k}$, $\bar{v}_k\triangleq \frac{1}{m} \sum_{i=1}^m v_{i,k}$, and $z_k\triangleq prox_{\alpha,h}(\bar{v}_k)$.
The following lemma states that the update of the average variable is viewed as an inexact centralized proximal gradient algorithm with the errors controlled by multiple communications at each iteration. 
\begin{lemma}\label{compact_lemma}
	Suppose Assumptions \ref{pro_assum} and \ref{net_assum} hold. The average variable satisfies
	\begin{align}\label{compact_m}
	\bar{x}_{k+1}\in &prox_{\alpha,h}^{\varepsilon_{k+1}} \big(\bar{x}_k-\alpha [\nabla g(\bar{x}_k)+e_{k+1}]\big),\\
	e_{k+1}=&\frac{1}{m}\sum_{i=1}^m \big(\nabla g_i(x_{i,k})-\nabla g_i(\bar{x}_k)\big),\notag\\
	 \varepsilon_{k+1}=&\|\bar{x}_{k+1}-z_{k+1}\|\big(G_h+\frac{1}{\alpha}\|z_{k+1}-\bar{v}_{k+1}\|\big)+\frac{1}{2\alpha}\|\bar{x}_{k+1}-z_{k+1}\|^2,\notag
	\end{align}
	where the inexact proximal operator ${\rm prox}_{\alpha,h}^{\varepsilon}(\cdot)$ is defined in (\ref{inprox_def}), $\nabla g(\bar{x}_k)\triangleq \frac{1}{m}\sum_{i=1}^m \nabla g_i(\bar{x}_k)$, $G_h$ is defined in Assumption \ref{pro_assum}(d), and error sequences $\{e_k\}$ and $\{\varepsilon_k\}$ satisfy
	\begin{subequations}
		\begin{align}
			&\left\|e_{k+1}\right\|\leq \frac{L}{m} \sum_{i=1}^m\left\|x_{i,k}-\bar{x}_k\right\|,\label{ek}\\
			&\varepsilon_{k+1}\leq \frac{2G_h}{m} \sum_{i=1}^m \left\|v_{i,k}-\bar{v}_k\right\|+\frac{1}{2\alpha}\big(\frac{1}{m}\sum_{i=1}^m \left\|v_{i,k}-\bar{v}_k\right\|\big)^2.\label{varepk}
		\end{align}\label{error}
	\end{subequations}
\end{lemma}
\begin{proof}
	 By taking the average of \eqref{q_update} and \eqref{v_update}, 
	 \begin{align}
	 \bar{v}_{k+1}=\bar{x}_k-\alpha (\nabla g(\bar{x}_k)+e_{k+1}),
	 \end{align}
	 where 
	$$e_{k+1}=\frac{1}{m}\sum_{i=1}^m [\nabla g_i(x_{i,k})-\nabla g_i(\bar{x}_k)].$$
	Because of the Lipschitz-continuity of the gradient of $g_i(x)$,
	$$\|e_{k}\|\leq \frac{L}{m}\sum_{i=1}^m \|x_{i,k}-\bar{x}_k\|.$$
	\par Let 
	$$z_{k+1}={\rm prox}_{\alpha,h}(\bar{v}_{k+1})={\rm argmin}_x \big\{h(x)+\frac{1}{2\alpha}\|x-\bar{v}_{k+1}\|^2\big\}$$
	denote the result of the exact proximal operator. In addition, $\bar{x}_{k+1}=\frac{1}{m}\sum_{i=1}^m x_{i,k+1}=\frac{1}{m}\sum_{i=1}^m {\rm prox}_{\alpha,h}(v_{i,k+1})$. Then, the result of the proximal operator in the distributed algorithm can be seen as an approximation of $z_{k+1}$. We next relate $z_{k+1}$ and $\bar{x}_{k+1}$ by formulating the latter as an inexact proximal operator with error $\varepsilon_{k+1}$. A simple algebraic expansion gives
	\begin{align*}
		&h(\bar{x}_{k+1})+\frac{1}{2\alpha}\|\bar{x}_{k+1}-\bar{v}_{k+1}\|^2\\
	\leq&h(z_{k+1})+G_h\|\bar{x}_{k+1}-z_{k+1}\|+\frac{1}{2\alpha}\Big\{\|z_{k+1}-\bar{v}_{k+1}\|^2+2\langle z_{k+1}-\bar{v}_{k+1},\bar{x}_{k+1}-z_{k+1}\rangle+\|\bar{x}_{k+1}-z_{k+1}\|^2\Big\}\\
	=&\min_{z\in\mathbb{R}^d}\big\{h(z)+\frac{1}{2\alpha}\|z-\bar{v}_{k+1}\|^2\big\}+\|\bar{x}_{k+1}-z_{k+1}\|\big(G_h+\frac{1}{\alpha}\|z_{k+1}-\bar{v}_{k+1}\|\big)+\frac{1}{2\alpha}\|\bar{x}_{k+1}-z_{k+1}\|^2, 
	\end{align*}
	where in the inequality, we used the convexity of $h(x)$ and the bound on the subgradient $\partial h(\bar{x}_{k+1})$ to obtain $h(\bar{x}_{k+1})\leq h(z_{k+1})+G_h\|\bar{x}_{k+1}-z_{k+1}\|$, and in the equality, we used the fact that by definition, $z_{k+1}$ is the optimizer of $h(x)+\frac{1}{2\alpha}\|x-\bar{v}_{k+1}\|^2$.
	\par With this expression, we can write
	$$\bar{x}_{k+1}\in {\rm prox}_{\alpha,h}^{\varepsilon_{k+1}}(\bar{v}_{k+1}),$$
	where $$\varepsilon_{k+1}=\|\bar{x}_{k+1}-z_{k+1}\|\big(G_h+\frac{1}{\alpha}\|z_{k+1}-\bar{v}_{k+1}\|\big)+\frac{1}{2\alpha}\|\bar{x}_{k+1}-z_{k+1}\|^2.$$
	\par By definition, $z_{k+1}={\rm prox}_{\alpha,h}(\bar{v}_{k+1})$ also implies $\frac{1}{\alpha}(\bar{v}_{k+1}-z_{k+1})\in \partial h(z_{k+1})$, and therefore its norm is bounded by $G_h$. As a result,
	$$\varepsilon_{k+1}\leq 2 G_h\|\bar{x}_{k+1}-z_{k+1}\|+\frac{1}{2\alpha}\|\bar{x}_{k+1}-z_{k+1}\|^2.$$
	\par Combined with the nonexpensiveness of the proximal operator,
	\begin{align*}
	\|\bar{x}_{k+1}-z_{k+1}\|\leq &\frac{1}{m}\sum_{i=1}^m \|{\rm prox}_{\alpha,h}(v_{i,k+1})-{\rm prox}_{\alpha,h}(\bar{v}_{k+1})\|\\
	\leq & \frac{1}{m} \sum_{i=1}^m \|v_{i,k+1}-\bar{v}_{k+1}\|,
	\end{align*}
	we obtain the desired results.
\end{proof}
\par The next lemma shows that polynomial-geometric sequences are summable, which is vital for the convergence analysis of error sequences.
\begin{lemma}\label{poly_lemma}\cite[Proposition 3]{proximal-convex}
	Let $\gamma\in (0,1)$, and let
	$$P_{k,N}=\{c_N k^N+\cdots+c_1 k+c_0 | c_j\in \mathbb{R}, j=0,\cdots,N\}$$
	denote the set of all $N$-th order polynomials of $k$, where $N \in \mathbb{N}$. Then for every polynomial $p_{k}\in P_{k,N}$,
	$$\sum_{k=1}^{\infty} p_{k} \gamma^k<\infty.$$
\end{lemma}
\par The result of this Lemma for $P_{k,N}=k^N$ will be particularly useful for the analysis in the following sections. Hence, we make the definition
\begin{align}\label{kn_poly}
S_{N}^{\gamma}\triangleq \sum_{k=1}^{\infty} k^N \gamma^k<\infty.
\end{align}
\par Before proving the summability of error sequences $\{\|e_k\|\}$ and $\{\varepsilon_k\}$, recursive expressions of the generated iterative variables are given in the next proposition.
\begin{proposition}\label{rec_ite}
	Under Assumptions \ref{pro_assum} and \ref{net_assum}, for each iteration $k\geq 2$,
	\begin{itemize}
		\item[(a)] $\sum_{i=1}^m \|q_{i,k+1}\|\leq \sum_{i=1}^m \|q_{i,k}\|+\alpha m (G_g+G_h)$,
		\item[(b)] $\sum_{i=1}^m \|x_{i,k}-x_{i,k-1}\|\leq
 2m {\Gamma} \sum_{l=1}^{k-1} \gamma^l \sum_{i=1}^m \|q_{i,l}\|+(k-1)\alpha m(G_g+G_h)$,
		\item[(c)] $\|x_{i,k}-\bar{x}_k\|\leq 2\Gamma \gamma^k \sum_{i=1}^m\|q_{i,k}\|$.
	\end{itemize}
\end{proposition}
\begin{proof}
	By (\ref{x_update}) and Proposition \ref{prox_proposition} (c), there exists $z_{i,k}\in \partial h(x_{i,k})$ such that
	\begin{align}\label{x_up}
	x_{i,k}=v_{i,k}-\alpha z_{i,k}.
	\end{align}
	Since function $h$ has bounded subgradients by Assumption \ref{pro_assum}(d), 
	\begin{align}\label{x_bound}
	\left\|x_{i,k}-v_{i,k}\right\|\leq \alpha G_h.
	\end{align}
	\par (a) Taking norm of (\ref{q_update}) and summing over $i$,
	\begin{align}\label{q_bound}
	\sum_{i=1}^m \left\|q_{i,k}\right\|=&\sum_{i=1}^m \left\|x_{i,k-1}-\alpha \nabla {g}_i (x_{i,k-1})\right\|\notag \\
	\leq &\sum_{i=1}^m \left\|x_{i,k-1}\right\|+\alpha m G_g.
	\end{align}
	It follows from (\ref{x_bound}) and $\left\|x_{i,k-1}\right\|-\left\|v_{i,k-1}\right\|\leq\left\|x_{i,k-1}-v_{i,k-1}\right\|$ that $\left\|x_{i,k-1}\right\|\leq \left\|v_{i,k-1}\right\|+\alpha G_h$. Since $v_{i,k-1}$ is a convex combination of $\{q_{j,k-1}\}_{j=1}^m$ by (\ref{v_update}), 
	\begin{align}\label{sum_vq}
	\sum_{i=1}^m \left\|v_{i,k-1}\right\|\leq \sum_{i=1}^m \left\|q_{i,k-1}\right\|.
	\end{align}
	Substituting the above two inequalities in (\ref{q_bound}),
	\begin{align*}
	\sum_{i=1}^m \left\|q_{i,k}\right\|\leq \sum_{i=1}^m \left\|q_{i,k-1}\right\|+\alpha m (G_g+G_h).
	\end{align*}
	\par (b) By (\ref{x_up}) and the proposed algorithm (\ref{distri_m}), 
	$x_{i,k}=v_{i,k}-\alpha z_{i,k}= \sum_{j=1}^m \lambda_{ij,k}\big(x_{j,k-1}-\alpha \nabla g_j (x_{j ,k-1})\big)-\alpha z_{i,k}$.
	Then,
	\begin{align}\label{x_gapnorm}
	&\sum_{i=1}^m \|x_{i,k}-x_{i,k-1}\|\notag\\
	\leq &\sum_{i=1}^m\sum_{j=1}^m \lambda_{ij,k}\|x_{j,k-1}-x_{i,k-1}\|+\alpha m (G_g+G_h).
	\end{align}
	Next, consider the term $\sum_{i=1}^m\sum_{j=1}^m \lambda_{ij,k}\|x_{j,k-1}-x_{i,k-1}\|$. By the nonexpansiveness of the proximal operator, 
	\begin{align}\label{non_expen}
	\|x_{j,k-1}-x_{i,k-1}\|\leq \|v_{j,k-1}-v_{i,k-1}\|.
	\end{align}
	 In addition, the bound of the distance between iterates $v_{i,k}$ and $\bar{v}_k$ satisfies
	\begin{align}\label{v_gapnorm}
	\|v_{i,k}-\bar{v}_k\|=& \Big\|\sum_{j=1}^m \lambda_{ij,k} q_{j,k}-\frac{1}{m}q_{j,k}\Big\|\notag \\
	\leq & \sum_{j=1}^m \left|\lambda_{ij,k}-\frac{1}{m}\right|\|q_{j,k}\|\notag \\
	\leq & \Gamma \gamma^k \sum_{j=1}^m \|q_{j,k}\|,
	\end{align}
	where the last inequality follows from Proposition 1 \cite{distri-nedich}, and $\Gamma = 2 \frac{1+\eta^{-B_0}}{1-\eta^{B_0}}$, $\gamma = (1-\eta^{B_0})^{\frac{1}{B_0}}$, $B_0=(m-1)B$, $\eta$ is the lower bound in Assumption \ref{net_assum}(b), $B$ is the intercommunication interval bound in Assumption \ref{net_assum}(c). Then, 
	\begin{align}\label{vsub}
	&\sum_{i=1}^m\sum_{j=1}^m \lambda_{ij,k}\|v_{j,k-1}-v_{i,k-1}\|
	\notag\\
	\leq & \sum_{i=1}^m\sum_{j=1}^m \lambda_{ij,k}\big(\|v_{i,k-1}-\bar{v}_{k-1}\|+\|v_{j,k-1}-\bar{v}_{k-1}\|\big)\notag\\
	\leq & 2m \Gamma \gamma^{k-1}\sum_{j=1}^m\|q_{j,k-1}\|,
	\end{align}
	where the last inequality follows from (\ref{v_gapnorm}). Substituting \eqref{non_expen} and \eqref{vsub} to (\ref{x_gapnorm}), 
	\begin{align}
	&\sum_{i=1}^m\|x_{i,k}-x_{i,k-1}\|\notag \\
	\leq& 2m \Gamma \gamma^{k-1}\sum_{j=1}^m\|q_{j,k-1}\|+\alpha m (G_g+G_h) \label{xkgap}\\
	\leq & 2m {\Gamma} \sum_{l=1}^{k-1} \gamma^l \sum_{i=1}^m \|q_{i,l}\|+(k-1)\alpha m(G_g+G_h).\notag
	\end{align}
	\par (c) By the definition of $\bar{x}_k\triangleq \frac{1}{m}\sum_{j=1}^m x_{j,k}$,
	\begin{align*}
	&\sum_{i=1}^m \|x_{i,k}-\frac{1}{m}\sum_{j=1}^m x_{j,k}\|\\
	=&\sum_{i=1}^m \|\frac{1}{m}\sum_{j=1}^m( x_{i,k}- x_{j,k})\|\\
	\leq& \frac{1}{m}\sum_{i=1}^m \sum_{j=1}^m\|v_{i,k}-v_{j,k}\|\\
	\leq & 2\Gamma \gamma^k \sum_{i=1}^m\|q_{i,k}\|,
	\end{align*}
where the first inequality is from nonexpansiveness of the proximal operator and the last inequality follows from \eqref{v_gapnorm}.
\end{proof}
By Proposition \ref{rec_ite} and Lemma 1 in \cite{proximal-convex}, there is a polynomial bound on $\sum_{i=1}^m \|q_{i,k}\|$, which is stated in the following lemma. The proof is omitted since it is similar to the proof of Lemma 1 in \cite{proximal-convex}.
\begin{lemma}\label{qbound_lemma}
		Under Assumptions \ref{pro_assum} and \ref{net_assum}, for the proposed algorithm (\ref{distri_m}), there exist non-negative scalars $C_q=C_q(q_{1,2},\cdots,q_{m,2})$, $C_q^1=C_q^1(m,\Gamma,C_q,C_q^{2})$, $C_q^{2}=C_q^{2}(m,\alpha,G_g,G_h)$ such that for iteration $k\geq 2$,
	\begin{align*}
	\sum_{i=1}^m\|q_{i,k}\|\leq C_q+C_q^1 k+C_q^2 k^2.
	\end{align*}
\end{lemma}
\begin{proof}
We proceed by induction on $k$. First, we show that the result holds for $k=2$ by choosing $C_q=\sum_{i=1}^m\|q_{i,2}\|$. It suffices to show that, given the initial points $x_{i,1}$, $\sum_{i=1}^m\|q_{i,2}\|$ is bounded.
\par Indeed, by \eqref{q_bound},
$$\sum_{i=1}^m\|q_{i,2}\|\leq \sum_{i=1}^m \|x_{i,1}\|+\alpha m G_g\leq \sum_{i=1}^m \|q_{i,1}\|+\alpha m (G_g+G_h)<\infty,$$
where the second inequality holds because of \eqref{x_up} and \eqref{sum_vq}. Therefore, $C_q=\sum_{i=1}^m\|q_{i,2}\|<\infty$ is a valid choice.
\par We scale Proposition \ref{rec_ite}(a) to 
\begin{align}\label{scale_a}
\sum_{i=1}^m \left\|q_{i,k+1}\right\|\leq \sum_{i=1}^m \left\|q_{i,k}\right\|+\alpha m (G_g+G_h)+\sum_{i=1}^m \|x_{i,k}-x_{i,k-1}\|.
\end{align}
 Now suppose the result holds for some positive integer $k\geq 2$. We show that it also holds for $k+1$.
\par Substituting the induction hypothesis for $k$ into Proposition \ref{rec_ite}(b), we have 
$$\sum_{i=1}^m \|x_{i,k}-x_{i,k-1}\|\leq 2m\Gamma \sum_{l=1}^{k-1}\gamma^l (C_q+C_q^1 l+C_q^2 l^2)+(k-1)\alpha m (G_g+G_h).$$
\par By Lemma \ref{poly_lemma} and \eqref{kn_poly}, there exist constants $S_{0}^{\gamma}, S_{1}^{\gamma}, S_2^{\gamma}$ such that 
$$\sum_{l=1}^{\infty} \gamma^l (C_q+C_q^1 l+C_q^2 l^2)\leq C_q S_0^{\gamma}+C_q^1S_1^{\gamma}+C_q^2 S_2^{\gamma}.$$
\par  Then, by induction hypothesis,
$$\sum_{i=1}^m \|q_i^{k+1}\|\leq C_q+C_q^1 k+C_q^2 k^2+\alpha m (G_g+G_h)+2m\Gamma (C_q S_0^{\gamma}+C_q^1S_1^{\gamma}+C_q^2 S_2^{\gamma})+(k-1)\alpha m (G_g+G_h).$$
Comparing coefficients, we see that the right-hand side can be bounded by $C_q+C_q^1 (k+1)+C_q^2(k+1)^2$ if $\alpha m (G_g+G_h)<2C_q^2$ for the coefficient of $k$, and $2m\Gamma (C_q S_0^{\gamma}+C_q^1S_1^{\gamma}+C_q^2 S_2^{\gamma})\leq C_q^1+C_q^2$ for the constant coefficient. Therefore, the induction hypothesis holds for $k+1$ if we take
\begin{align*}
C_q&=\sum_{i=1}^m \|q_{i,2}\|,\\
C_q^1&=\frac{2m\Gamma C_q S_0^{\gamma}+(2m\Gamma S_2^{\gamma}-1)C_q^2}{2m\Gamma S_1^{\gamma}-1},\\
C_q^2&=\frac{\alpha m}{2}(G_g+G_h).
\end{align*}
\end{proof}
Now, by the Lemma \ref{qbound_lemma} and Proposition \ref{rec_ite}, the boundedness of summabilities (defined in Lemma \ref{poly_lemma}) of error sequences $\{\|e_k\|\}$ and $\{\varepsilon_k\}$ is proved in the following proposition. 
\begin{proposition}\label{e_sum_pro}
		Under Assumptions \ref{pro_assum} and \ref{net_assum}, for sequences $\{e_{k}\}$ and $\{\varepsilon_k\}$ defined in (\ref{error}), $\sum_{k=1}^{\infty} \left\|e_k\right\|<\infty$, $\sum_{k=1}^{\infty} {\varepsilon_k}<\infty$ and $\sum_{k=1}^{\infty} \sqrt{\varepsilon_k}<\infty$.
\end{proposition}
\begin{proof}
	By Lemma \ref{poly_lemma}, it suffices to show that these error sequences are polynomial-geometric sequence.
	\par (a) By Proposition \ref{rec_ite}(c),
	\begin{align}\label{xsubbar}
	\frac{1}{m} \sum_{i=1}^m \|x_{i,k}-\bar{x}_k\| \leq  2 \Gamma \gamma^k \sum_{i=1}^m \|q_{i,k}\|.
	\end{align} 
	 It follows from (\ref{ek}) and \eqref{xsubbar} that
	\begin{align}
	\|e_{k+1}\|&\leq 2L\Gamma \gamma^{k} \sum_{i=1}^m\|q_{i,k}\|. \label{eksumma}
	\end{align}
	In addition, by Lemma \ref{qbound_lemma}, there exists  $\sum_{i=1}^m\|q_{i,k}\|\leq C_q+C_q^1 k+C_q^2 k^2$ such that
	\begin{align}
	\|e_k\|\leq 2mL\Gamma \gamma^{k-1} \big(C_q+C_q^1 (k-1)+C_q^2 (k-1)^2\big),
	\end{align}
	which implies that $\{\|e_k\|\}$ is a polynomial-geometric sequence.
	\par (b) It follows from (\ref{v_gapnorm}), (\ref{varepk}) and Lemma \ref{qbound_lemma} that 
	\begin{align*}
	\varepsilon_k\leq &2G_h\Gamma \gamma^k(C_q+C_q^1 k+C_q^2 k^2)+\frac{1}{2\alpha}\big[\Gamma \gamma^k (C_q+C_q^1 k+C_q^2 k^2)\big]^2.
	\end{align*}
	Using the fact that $\sqrt{a+b}\leq \sqrt{a}+\sqrt{b}$ for all nonnegative real numbers $a,b$, 
	\begin{align*}
	\sqrt{\varepsilon_k}\leq & \sqrt{2G_h \Gamma} \sqrt{\gamma^k}\big(\sqrt{C_q}+\sqrt{C_q^1}k+\sqrt{C_q^2}k\big)+\frac{1}{\sqrt{2\alpha}} \Gamma \gamma^k(C_q+C_q^1 k+C_q^2 k^2).
	\end{align*} 
	Therefore, both sequences $\{\varepsilon_{k}\}$ and $\{\sqrt{\varepsilon_{k}}\}$ are polynomial-geometric sequences.
\end{proof}
Next, we prove that all local variables achieve consensus and converge to the average.
\begin{theorem}\label{ave_theo}
	Under Assumptions \ref{pro_assum} and \ref{net_assum}, $\lim_{k\to \infty} \|x_{i,k}-\bar{x}_k\|=0$ for all $i=1,\cdots,m$.
\end{theorem}
\begin{proof}
	By Proposition \ref{rec_ite} (c), 
	 $$\|x_{i,k}-\bar{x}_k\|\leq 2\Gamma \gamma^k \sum_{i=1}^m\|q_{i,k}\|.$$
	By Lemma \ref{poly_lemma} and Lemma \ref{qbound_lemma}, $\sum_{k=1}^\infty \|x_{i,k}-\bar{x}_k\|$ is bounded. Then, by monotone convergence theorem and Cauchy condition, we obtain
	{$$\lim_{k\to \infty} \|x_{i,k}-\bar{x}_k\|=0.$$}
	Therefore, local variables achieve consensus and converge to the average $\bar{x}_k$ as $k\to \infty$. 
\end{proof}
\par The next vital lemma characterizes $\partial_{\varepsilon_k}h(x_k)$, which is the $\varepsilon_k$-subdifferential of $h$ at $x_k$. The proof has been studied in Lemma 2 of  \cite{Schmidt2012}.
\begin{lemma}\label{pk_lemma}
	If $\bar{x}_k$ is an $\varepsilon_{k}$-optimal solution to \eqref{prox_ope} in the sense of \eqref{inprox_def} with $y=\bar{x}_{k-1}-\alpha (\nabla g(\bar{x}_{k-1})+e_k)$, then there exists $p_k\in \mathbb{R}^n$ such that $\|p_k\|\leq \sqrt{2\alpha \varepsilon_k}$ and 
	$$\frac{1}{\alpha}(\bar{x}_{k-1}-\bar{x}_{k}-\alpha\nabla g(\bar{x}_{k-1})-\alpha e_{k}-p_{k})\in \partial_{\varepsilon_{k}}h(\bar{x}_{k}).$$
\end{lemma}
Now, motivated by the work \cite{GuWHH18}, we are ready to discuss the convergence performance of the proposed algorithm. For a convex and closed set $\mathcal{K}\subset \mathbb{R}^n$, define $m(\mathcal{K})=\min_{x\in \mathcal{K}} \|x\|$. 

\begin{theorem}\label{gra_theo}
		Suppose Assumptions \ref{pro_assum} and \ref{net_assum} hold. $ \lim_{k\to \infty} m\big(\nabla g(\bar{x}_k)+\partial_{\varepsilon_{k+1}}h(\bar{x}_k)\big)=0$ for distributed proximal gradient algorithm \eqref{distri_m}.
\end{theorem}
\begin{proof}
	Let $ \mathbf{g}_k=\nabla g(\bar{x}_k)+ e_{k+1}$. It follows from \eqref{compact_m} and the definition of inexact proximal operator \eqref{inprox_def} that
	\begin{align*}
	&\frac{1}{2\alpha}\|\bar{x}_{k+1}-\bar{x}_k+\alpha  \mathbf{g}_k\|^2+h(\bar{x}_{k+1}) \leq \\
	& \varepsilon_{k+1} +\min_{x} \big\{ \frac{1}{2\alpha} \|x-\bar{x}_k+\alpha \mathbf{g}_k\|^2+h(x)\big\}.	
	\end{align*}
	Equivalently,
	\begin{align*}
	&\frac{1}{2\alpha} \|\bar{x}_{k+1}-\bar{x}_k\|^2+\langle\bar{x}_{k+1}-\bar{x}_k,\mathbf{g}_k \rangle+h(\bar{x}_{k+1})\leq\\
	& \varepsilon_{k+1}+\min_x \big\{ \frac{1}{2\alpha} \|x-\bar{x}_k\|^2+\langle x-\bar{x}_k,\mathbf{g}_k \rangle+h(x)\big\}.
	\end{align*}
	Take $x = \bar x_k$ in the right-hand side of the above equation. We obtain
    \begin{align}\label{prox_ineq}
	\left<\mathbf{g}_k,\bar{x}_{k+1}-\bar{x}_k\right>&+\frac{1}{2\alpha} \left\|\bar{x}_{k+1}-\bar{x}_k\right\|^2+h(\bar{x}_{k+1})
    \leq  h(\bar{x}_k)+\varepsilon_{k+1}.
    \end{align}
    By \eqref{lfsmooth} and \eqref{prox_ineq},
	\begin{align*}
	f(\bar{x}_{k+1})=&g(\bar{x}_{k+1})+h(\bar{x}_{k+1})\\
	\leq& g(\bar{x}_{k})+\left< \nabla g(\bar{x}_k), \bar{x}_{k+1}-\bar{x}_k\right>+\frac{L}{2}\left\|\bar{x}_{k+1}-\bar{x}_k\right\|^2\\
	&+h(\bar{x}_k)-\left<\nabla g(\bar{x}_k)+e_{k+1}, \bar{x}_{k+1}-\bar{x}_k\right>\\
    &-\frac{1}{2\alpha} \left\|\bar{x}_{k+1}-\bar{x}_k\right\|^2+\varepsilon_{k+1}\\
	=&f(\bar{x}_k)-(\frac{1}{2\alpha}-\frac{L}{2})\left\|\bar{x}_{k+1}-\bar{x}_k\right\|^2\\
	&+\varepsilon_{k+1}-\left<e_{k+1},\bar{x}_{k+1}-\bar{x}_k\right>,
	\end{align*}
	where the first inequality is due to the convexity of function $g$ and the proximal operator relationship (\ref{prox_ineq}).
	By summing the inequality over $k=1,\cdots,T$,
	\begin{align*}
	f(\bar{x}_{T+1})\leq& f(\bar{x}_1)-(\frac{1}{2\alpha}-\frac{L}{2})\sum_{k=1}^T \left\|\bar{x}_{k+1}-\bar{x}_k\right\|^2\\
	&+\sum_{k=1}^T \varepsilon_{k+1}+\sum_{k=1}^T\left\|e_{k+1}\right\|\left\|\bar{x}_{k+1}-\bar{x}_k\right\|.
	\end{align*}
    By rearranging the terms,
    \begin{align}\label{barxplusx}
    &\sum_{k=1}^T \left\|\bar{x}_{k+1}-\bar{x}_k\right\|^2\notag\\
    \leq &\frac{1}{\frac{1}{2\alpha}-\frac{L}{2}}\big(f(\bar{x}_1)-f(
    \bar{x}_{T+1})\big)+\frac{1}{\frac{1}{2\alpha}-\frac{L}{2}}\sum_{k=1}^T
    \varepsilon_{k+1}\notag\\
    &+\frac{1}{\frac{1}{2\alpha}-\frac{L}{2}}\sum_{k=1}^T\left\|e_{k+1}\right\|\left\|\bar{x}_{k+1}-\bar{x}_k\right\|,
    \end{align}
    where $\frac{1}{\frac{1}{2\alpha}-\frac{L}{2}}>0$ because the step-size satisfies $\alpha<\frac{1}{L}$.
    \par Then, consider the last term $\sum_{k=1}^T\left\|e_{k+1}\right\|\left\|\bar{x}_{k+1}-\bar{x}_k\right\|$.It follows from the inequality \eqref{xkgap},
    \begin{align}\label{barxkgap}
    &\|\bar{x}_{k+1}-\bar{x}_k\|\notag\\
    = &\frac{1}{m}\sum_{i=1}^m\|x_{i,k+1}-x_{i,k}\|\notag \\
    \leq &  2 {\Gamma} \big( \gamma^k \sum_{j=1}^m \|q_{j,k}\|\big)+ \alpha (G_g+G_h).
    \end{align}
     Then, by (\ref{eksumma}) and (\ref{barxkgap}),
    \begin{align}
    \|e_{k+1}\|&\|\bar{x}_{k+1}-\bar{x}_{k}\|\leq 4mL\Gamma^2 \gamma^{2k}\big(\sum_{i=1}^m\|q_{i,k}\|\big)^2\notag \\
    &+2\alpha mL\Gamma (G_g+G_h)\gamma^k \sum_{i=1}^m \|q_{i,k}\|,
    \end{align}
    where $\sum_{i=1}^m\|q_{i,k}\|\leq C_q+C_q^1 k+C_q^2 k^2$ and the right hand side is a polynomial-geometric sequence. By Lemma \ref{poly_lemma}, 
    \begin{align}\label{leqinf}
    \sum_{k=1}^T \|e_{k+1}\|\|\bar{x}_{k+1}-\bar{x}_{k}\|<\infty.
    \end{align}
	\par In addition, by Proposition \ref{e_sum_pro}, $\sum_{k=1}^T \varepsilon_{k+1}<\infty$. 
	It follows from (\ref{barxplusx}) and (\ref{leqinf}) that $\sum_{k=1}^T \left\|\bar{x}_{k+1}-\bar{x}_k\right\|^2< \infty.$ It follows from monotone convergence theorem and Cauchy condensation test\upcite{conver_cauchy} that
	\begin{align}\label{xsublimt}
	\lim_{k\to \infty} \left\|\bar{x}_{k+1}-\bar{x}_k\right\|=0.
	\end{align}
	By Lemma \ref{pk_lemma}, there exists $p_{k+1}$ such that $\left\|p_{k+1}\right\|\leq \sqrt{2\alpha\varepsilon_{k+1}}$ and
	\begin{align*}
	0\in &\frac{1}{\alpha}(\bar{x}_k-\bar{x}_{k+1}-p_{k+1})-\nabla g(\bar{x}_k)-e_{k+1}+\nabla g(\bar{x}_{k+1})\notag \\
    &-\nabla g(\bar{x}_{k+1})-\partial_{\varepsilon_{k+1}}h(\bar{x}_{k+1}).
	\end{align*}
	Then,
	\begin{align}\label{gr_eq}
	&\frac{1}{\alpha}(\bar{x}_k-\bar{x}_{k+1}-p_{k+1})-\nabla g(\bar{x}_k)-e_{k+1}+\nabla g(\bar{x}_{k+1})\notag \\
	&\in \nabla g(\bar{x}_{k+1})+\partial_{\varepsilon_{k+1}}h(\bar{x}_{k+1}).
	\end{align}
	The left hand side of \eqref{gr_eq} satisfies 
	\begin{align}\label{gr_normeq}
	&\left\|\frac{1}{\alpha}(\bar{x}_k-\bar{x}_{k+1}-p_{k+1})-\nabla g(\bar{x}_k)-e_{k+1}+\nabla g(\bar{x}_{k+1})\right\|\notag \\
    &\leq (\frac{1}{\alpha}+L)\left\|\bar{x}_k-\bar{x}_{k+1}\right\|+\sqrt{\frac{2\varepsilon_{k+1}}{\alpha}}+\|e_{k+1}\|,
	\end{align}
	where we utilize $\|p_k\|\leq \sqrt{2\alpha \varepsilon_{k}}$ in Lemma \ref{pk_lemma} and \eqref{l_assump} in Assumption \ref{pro_assum}.
	Then, by \eqref{xsublimt} and Proposition \ref{e_sum_pro}, 
	\begin{align}\label{norm2}
	&\lim_{k\to \infty}\big\|\frac{1}{\alpha}(\bar{x}_k-\bar{x}_{k+1}-p_{k+1}\!)-\nabla g(\bar{x}_k)-e_{k+1}+\nabla g(\bar{x}_{k+1})\big\|\notag\\
	&\leq \lim_{k\to \infty} ((\frac{1}{\alpha}+L)\left\|\bar{x}_k-\bar{x}_{k+1}\right\|+\sqrt{\frac{2\varepsilon_{k+1}}{\alpha}}+\|e_{k+1}\|)\notag\\
	&=0.
	\end{align}
	Therefore, it follows from (\ref{gr_eq}) and (\ref{norm2}) that
	$$\lim_{k\to \infty} m\big(\nabla g(\bar{x}_k)+\partial_{\varepsilon_{k+1}}h(\bar{x}_k)\big)=0.$$
\end{proof}
\begin{remark}
	Combining Theorems \ref{ave_theo} and \ref{gra_theo}, we obtain that for all agent $i$, the generated variable sequence $x_{i,k}$ converges to the set of critical points and there exists a subsequence of $x_{i,k}$ converging to one critical point of non-convex optimization problem \eqref{opti_pro}.
\end{remark}
Next, we consider the convergence rate of the proposed algorithm. By (\ref{gr_eq}) and (\ref{gr_normeq}),
\begin{align}\label{sum_norm}
&\frac{1}{T}\sum_{k=1}^T {\rm min}_{d_k\in \partial_{\varepsilon_k}h(\bar{x}_k)} \|\nabla g(\bar{x}_k)+d_k\| \notag \\
\leq&\frac{1}{T}\sum_{k=1}^T\big((\frac{1}{\alpha}+L)\|\bar{x}_k-\bar{x}_{k-1}\|+\sqrt{\frac{2\varepsilon_{k}}{\alpha}}+\|e_k\|\big).
\end{align}
Hence, we analyze $\frac{1}{T}\sum_{k=1}^T \|\bar{x}_k-\bar{x}_{k-1}\|^2$ to provide the convergence rate of \eqref{distri_m} in the non-convex setting. 
\par At first, we provide one related lemma, whose proof is provided in Lemma 1 of \cite{Schmidt2012}.
\begin{lemma}\label{uklemma}
	Assume that the non-negative sequence $u_k$ satisfies the following recursion for all $k\geq 1$:
	$$u_k^2\leq S_k+\sum_{i=1}^k \lambda_iu_i,$$
	with an increasing sequence $S_k$, $S_1\geq u_1^2$ and $\lambda_i\geq 0$. Then for all $k\geq 1$, 
	$$u_k\leq \frac{1}{2}\sum_{i=1}^k \lambda_i+\big(S_k+(\frac{1}{2}\sum_{i=1}^k \lambda_i)^2\big)^{\frac{1}{2}}.$$
\end{lemma}
\par Then, we are ready to analyze the convergence rate of $\frac{1}{T}\sum_{k=1}^T \|\bar{x}_k-\bar{x}_{k-1}\|^2$.
\begin{theorem}\label{xgap_norm_theo}
	With Assumptions \ref{pro_assum} and \ref{net_assum}, the convergence rate of the sequence  $\frac{1}{T}\sum_{k=1}^T\left\|\bar{x}_k-\bar{x}_{k-1}\right\|^2$ is $O(\frac{1}{T})$.
\end{theorem}
\begin{proof}
   Recall that
	$\frac{1}{\alpha}\big(\bar{x}_{k}-\bar{x}_{k+1}-\alpha \nabla g(\bar{x}_k)-\alpha e_{k+1}-p_{k+1}\big)\in \partial_{\varepsilon_{k+1}}h(\bar{x}_{k+1})$ in \eqref{gr_eq}. By \eqref{lfsmooth} and the definition of $\varepsilon_{k}$-subdifferential,
	\begin{align*}
	&f(\bar{x}_{k+1})=g(\bar{x}_{k+1})+h(\bar{x}_{k+1})\\
	\leq & g(\bar{x}_k)+\left<\nabla g(\bar{x}_k), \bar{x}_{k+1}-\bar{x}_k\right>+\frac{L}{2}\left\|\bar{x}_{k+1}-\bar{x}_k\right\|^2+h(\bar{x}_k)\\
	&-\Big\langle \nabla g(\bar{x}_k)+e_{k+1}+\frac{1}{\alpha}(\bar{x}_{k+1}-\bar{x}_k+p_{k+1}),\bar{x}_{k+1}-\bar{x}_k\Big\rangle+\varepsilon_{k+1}\\
	=&f(\bar{x}_k)-\frac{1}{\alpha}\left\|\bar{x}_{k+1}-\bar{x}_k\right\|^2+\frac{L}{2}\left\|\bar{x}_{k+1}-\bar{x}_k\right\|^2\\
    &-\Big\langle e_{k+1}+\frac{1}{\alpha}p_{k+1},\bar{x}_{k+1}-\bar{x}_k\Big\rangle+\varepsilon_{k+1}\\
	\leq &f(\bar{x}_k)-(\frac{1}{\alpha}-\frac{L}{2})\left\|\bar{x}_{k+1}-\bar{x}_k\right\|^2\\
	&+\Big(\left\|e_{k+1}\right\|+\sqrt{\frac{2\varepsilon_{k+1}}{\alpha}}\Big)\|\bar{x}_{k+1}-\bar{x}_k\|+\varepsilon_{k+1}.
	\end{align*}
\par By summing the above inequality over $k=1,\dots,T-1$, 
	\begin{align}\label{fsubfT}
f(\bar{x}_T)\leq& f(\bar{x}_0)-(\frac{1}{\alpha}-\frac{L}{2})\sum_{k=1}^{T-1} \left\|\bar{x}_{k+1}-\bar{x}_k\right\|^2+\sum_{k=1}^{T-1} \varepsilon_{k+1}\notag \\
&+\sum_{k=1}^{T-1}\Big(\left\|e_{k+1}\right\|+\sqrt{\frac{2\varepsilon_{k+1}}{\alpha}}\Big)\|\bar{x}_{k+1}-\bar{x}_k\|.
\end{align}
	Then, by rearranging and scaling,
	\begin{align*}
	\|\bar{x}_T\!-\!\bar{x}_{T-1}\|^2 \!\leq& \underbrace{\frac{1}{\frac{1}{\alpha}-\frac{L}{2}}\Big(f(\bar{x}_0)-f({x}^*)+\sum_{k=1}^T \varepsilon_{k}\Big)}_{S_T}\notag\\
&+\!\sum_{k=1}^T \underbrace{\!\frac{1}{\frac{1}{\alpha}-\frac{L}{2}}\Big(\left\|e_{k}\right\|\!+\!\sqrt{\frac{2\varepsilon_{k}}{\alpha}}\Big)\!}_{\lambda_k}\underbrace{\Big\|\bar{x}_k-\bar{x}_{k-1}\Big\|}_{u_k},
	\end{align*}
	where $x^*$ is the optimal solution of optimization problem. 
	\par Let $\bar{x}_{0}=\bar{x}_1$, then $u_1^2=0$ and $S_1\geq u_1^2$. By Lemma \ref{uklemma}, 
	\begin{align*}
	&\left\|\bar{x}_T-\bar{x}_{T-1}\right\|\\
	\leq & \frac{1}{2}\sum_{k=1}^T\lambda_k+\big(S_T+(\frac{1}{2}\sum_{k=1}^T \lambda_k)^2\big)^{\frac{1}{2}}\\
	=&\underbrace{\frac{1}{2}\sum_{k=1}^T \frac{1}{\frac{1}{\alpha}-\frac{L}{2}} \Big(\sqrt{\frac{2\varepsilon_{k}}{\alpha}}+\left\|e_{k}\right\|\Big)}_{A_T}\\
&+\Big(\frac{1}{\frac{1}{\alpha}-\frac{L}{2}} \big(f(\bar{x}_0)-f({x}^*)\big)+\underbrace{\frac{1}{\frac{1}{\alpha}-\frac{L}{2}}\sum_{k=1}^T \varepsilon_{k}}_{B_T}\\
&+\big(\underbrace{\frac{1}{2}\sum_{k=1}^T \frac{1}{\frac{1}{\alpha}-\frac{L}{2}}(\sqrt{\frac{2\varepsilon_{k}}{\alpha}}+\|e_{k}\|)}_{A_T}\big)^2\Big)^{\frac{1}{2}}.\\
	\end{align*}
	Because $A_k$ and $B_k$ are increasing sequences, $\forall k\leq T$,
	\begin{align}\label{xksubx}
	&\left\|\bar{x}_k-\bar{x}_{k-1}\right\|\notag\\
	\leq &A_T+\Big(\frac{1}{\frac{1}{\alpha}-\frac{L}{2}}(f(\bar{x}_0)-f({x}^*))+B_T+A_T^2\Big)^{\frac{1}{2}}\notag\\
	\leq & 2A_T+\sqrt{\frac{1}{\frac{1}{\alpha}-\frac{L}{2}}\big(f(\bar{x}_0)-f({x}^*)\big)}+\sqrt{B_T}.
	\end{align}
	By (\ref{fsubfT}) and (\ref{xksubx}), 
	\begin{align*}
	&\sum_{k=1}^T \|\bar{x}_k-\bar{x}_{k-1}\|^2\\
	\leq &\frac{1}{\frac{1}{\alpha}-\frac{L}{2}}\big(f(\bar{x}_0)-f({x}^*)\big)+B_T\\
	&+2A_T\Big(2A_T+\sqrt{\frac{1}{\frac{1}{\alpha}-\frac{L}{2}}\big(f(\bar{x}_0)-f({x}^*)\big)}+\sqrt{B_T}\Big)\\
	\leq &\Big(2A_T+\sqrt{\frac{1}{\frac{1}{\alpha}-\frac{L}{2}}\big(f(\bar{x}_0)-f({x}^*)\big)}+\sqrt{B_T}\Big)^2.
	\end{align*}
	Since $A_T$ and $B_T$ are upper bounded by Proposition \ref{e_sum_pro}, we obtain $\big\{\frac{1}{T}\sum_{k=1}^T\left\|\bar{x}_k-\bar{x}_{k-1}\right\|^2\big\}=O(1/T)$.
\end{proof}
By (\ref{sum_norm}) and Theorem \ref{xgap_norm_theo}, the proposed distributed algorithm has a convergence rate $O(\frac{1}{T})$ for the distributed non-convex optimization problem (\ref{opti_pro}).
\begin{remark}
We establish the convergence rate of the proposed algorithm with respect to the number of communication steps. Let $t$ be the total number of communication steps taken. Since the proposed algorithm takes $k$ communication steps in iteration $k$, the total number of communication steps to execute iterations $1,\cdots, T$ is $\sum_{k=1}^T k=\frac{T(T+1)}{2}$. Then, when $t$ communication steps are taken, the number of iterations completed is $T$ such that $\frac{T(T+1)}{2}<t$, or equivalently, $T=\lceil \frac{-1+\sqrt{1+8t}}{2}\rceil$. Hence, $\frac{1}{T}\sum_{k=1}^T {\rm min}_{d_k\in \partial_{\varepsilon_k}h(\bar{x}_k)} \|\nabla g(\bar{x}_k)+d_k\|=O(1/\sqrt{t})$, where $t$ is the total number of communication steps taken.
\end{remark}
\section{Simulation}\label{simulation}
We apply algorithm \eqref{distri_m} to learn the black-box binary classification problem\cite{Huang2019}, which is to find the optimal predictor $x\in \mathbb{R}^n$ by solving the following problem:
\begin{align}\label{black_box}
{\rm min}_{x\in \mathbb{R}^n} \frac{1}{m} \sum_{i=1}^m \frac{1}{1+{\rm exp}(l_i a_i'x)} +\lambda_1 \|x\|_1+\lambda_2\|x\|_2^2,
\end{align}
where $a_i\in \mathbb{R}^n$, $l_i \in \{-1,1\}$ and $\{a_i,l_i\}_{i=1}^m$ denotes the set of training samples. 
 \begin{table}[!htbp]
	\caption{Real data for black-box binary classification}\label{real_data}
	\centering
	\begin{tabular}{c|c|c|c}
		\hline
		datasets & \#samples & \#features & \#classes  \\
		\hline
		$a9a$  & 32561& 123 &2\\
		\hline
		$covtype.binary$&581012&54&2\\
		\hline
	\end{tabular}
\end{table}
In this experiment, we use the publicly available real datasets\footnote{$a9a$ and $covtype.binary$ are from the website www.csie.ntu.edu.tw/~cjlin/libsvmtools/datasets/.}, which are summarized in Table \ref{real_data}. The proposed distributed algorithm \eqref{distri_m} is applied over ten-agent networks to solve the problem \eqref{black_box}. Meanwhile, $\lambda_1=\lambda_2=5\times 10^{-4}$. For the underlying network, we take time-invariant undirected connected graphs and  periodic time-varying graphs, respectively. 
\par (1) Time-invariant undirected connected graphs: Over the undirected connected multi-agent networks, we apply \eqref{distri_m} (denoted by `meth1') and the algorithm in \cite{decen-nonconvex} (denoted by `meth2') to solve \eqref{black_box}. Both algorithms take the same constant step-size.
\par (2) Time-varying undirected graphs: Over periodic time-varying graphs, which are generated randomly and satisfy Assumption \ref{net_assum}, we apply \eqref{distri_m} (denoted by `meth1-s') and the algorithm in \cite{NEXT2016} (denoted by `meth3-s') to solve \eqref{black_box}. The algorithm `meth3-s' takes a diminishing step-size in \cite{NEXT2016}, and the proposed algorithm `meth1-s' takes a constant step-size. 
\begin{figure}
	\centering
	\includegraphics[width=6cm]{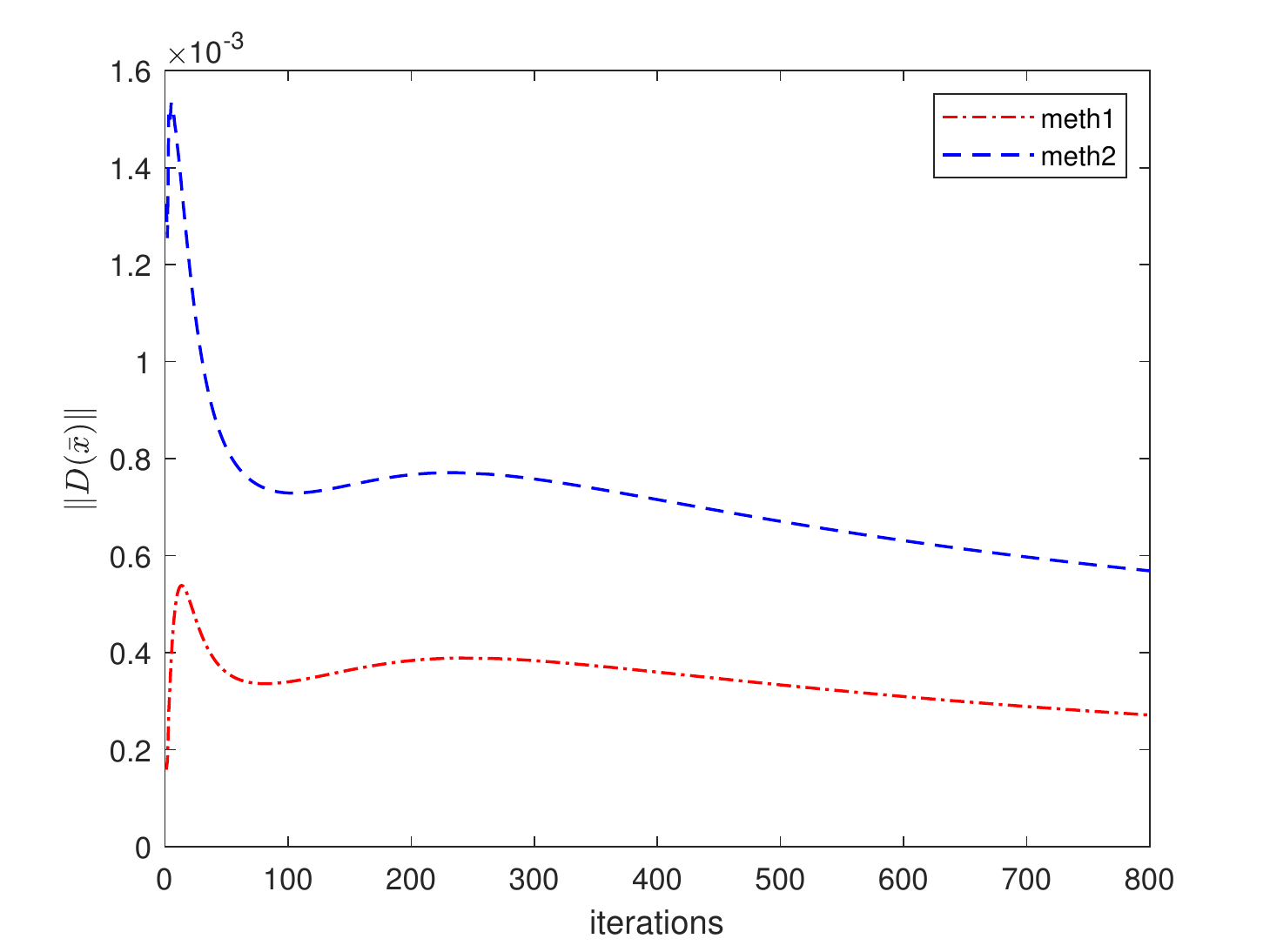}\\
	\caption{Consensus trajectories of local variables over time-invariant graphs}\label{a9a_fixed}
	\centering
\end{figure}
\begin{figure}
	\centering
	\subfigure[Convergence trajectories of different algorithms for real data $a9a$]{
		\includegraphics[width=8 cm, height =  3 cm]{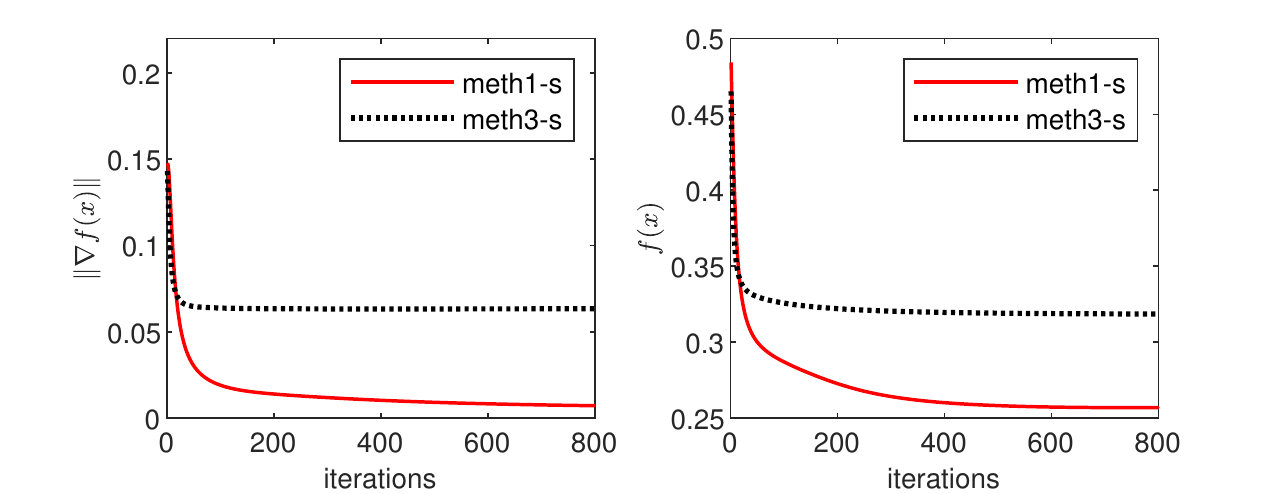}
		\label{a9a_fig}
	}
	\subfigure[Convergence trajectories of different algorithms for real data $covtype$]{
		\includegraphics[width=8 cm, height =  3cm]{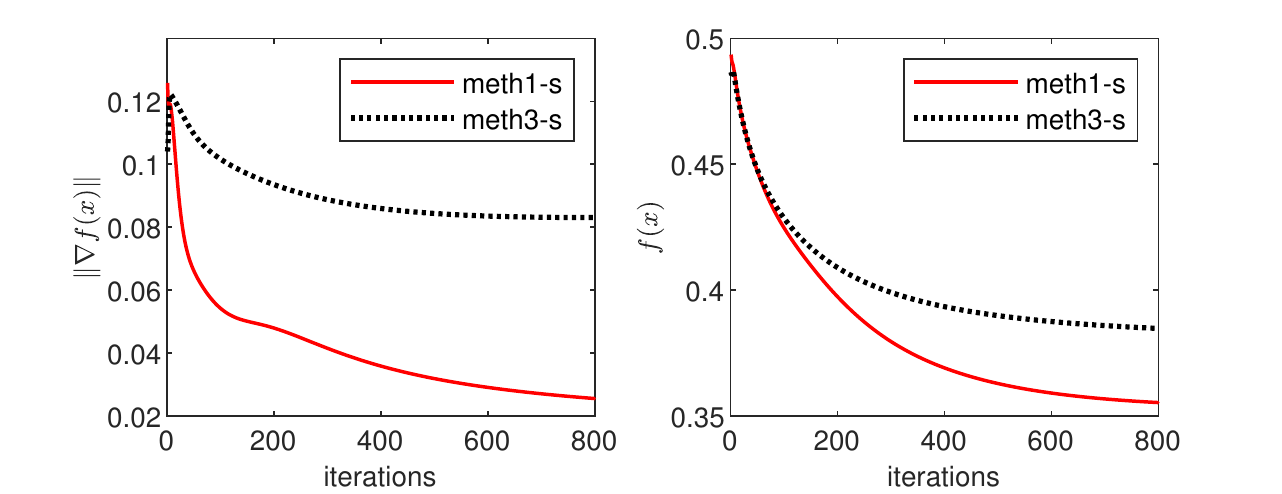}
		\label{covtype_fig}
	}
	\caption{Convergence results of different algorithms over time-varying graphs}
	\label{conver_fig}
\end{figure}

\par Define $D(\bar{x})=\sum_{i=1}^{10}x_i'\sum_{j=1}^{10} a_{ij}(x_i-x_j)$. For the dataset $a9a$ over time-invariant graphs, the trajectories of $D(\bar{x})$ are shown in Fig. \ref{a9a_fixed}, which implies that the variable eastimates of different agents achieve consensus. In addition, it is seen that the trajectory generated by \eqref{distri_m} owns a better consistent performance than those generated by \cite{decen-nonconvex}. For the datasetss $a9a$ and $covtype.binary$ over time-varying graphs, the convergence trajectories are shown in Figs. \ref{a9a_fig} and \ref{covtype_fig}, respectively. The trajectories of the norm of gradient converging to zero imply that variable estimates converge to critical points of opitmization problem \eqref{black_box}. It is seen that over time-varying undirected graphs, 'meth1-s' owns a better convergence performance than 'meth3-s' proposed in \cite{NEXT2016}. Hence, the proposed distributed algorithm \eqref{distri_m} is able to solve the binary classificaition problem \eqref{black_box} efficiently.
\section{Conclusion}\label{conclusion}
This paper has proposed a distributed proximal gradient algorithm for the large-scale non-smooth non-convex optimization problem over time-varying multi-agent networks. The proposed algorithm has made use of multi-step consensus stage to improve the consistent performance of the proposed distributed algorithm. In addition, this paper has provided complete and rigorous theoretical convergence analysis. In binary classification tests,  the proposed algorithm has a better convergence performance compared with existing distributed algorithms. One future research direction is to develop distributed stochastic algorithms for distributed non-convex optimization over multi-agent networks.

\bibliographystyle{ieeetran}
\bibliography{refer}

\begin{thebibliography}{10}
\providecommand{\url}[1]{#1}
\csname url@samestyle\endcsname
\providecommand{\newblock}{\relax}
\providecommand{\bibinfo}[2]{#2}
\providecommand{\BIBentrySTDinterwordspacing}{\spaceskip=0pt\relax}
\providecommand{\BIBentryALTinterwordstretchfactor}{4}
\providecommand{\BIBentryALTinterwordspacing}{\spaceskip=\fontdimen2\font plus
\BIBentryALTinterwordstretchfactor\fontdimen3\font minus
  \fontdimen4\font\relax}
\providecommand{\BIBforeignlanguage}[2]{{%
\expandafter\ifx\csname l@#1\endcsname\relax
\typeout{** WARNING: IEEEtran.bst: No hyphenation pattern has been}%
\typeout{** loaded for the language `#1'. Using the pattern for}%
\typeout{** the default language instead.}%
\else
\language=\csname l@#1\endcsname
\fi
#2}}
\providecommand{\BIBdecl}{\relax}
\BIBdecl

\bibitem{lasso_nonsmooth}
Y.~Hao, C.~Li, and R.~Wang, ``Sparse approximate solution of fitting surface to
  scattered points by {MLASSO} model,'' \emph{SCIENCE CHINA Mathematics},
  vol.~62, no.~7, p. 1319, 2018.

\bibitem{svm_nonsmooth}
A.~{Astorino} and A.~{Fuduli}, ``Nonsmooth optimization techniques for
  semisupervised classification,'' \emph{IEEE Transactions on Pattern Analysis
  and Machine Intelligence}, vol.~29, no.~12, pp. 2135--2142, 2007.

\bibitem{infor_nonsmooth}
G.~Li, S.~Song, and C.~Wu, ``Generalized gradient projection neural networks
  for nonsmooth optimization problems,'' \emph{SCIENCE CHINA Information
  Sciences}, vol.~53, no.~5, pp. 990--1005, 2010.

\bibitem{nonsmooth_tac}
N.~K. {Dhingra}, S.~Z. {Khong}, and M.~R. {Jovanović}, ``The proximal
  augmented lagrangian method for nonsmooth composite optimization,''
  \emph{IEEE Transactions on Automatic Control}, vol.~64, no.~7, pp.
  2861--2868, 2019.

\bibitem{nonsmooth_convex_ijrnc}
Q.~Wang, J.~Chen, X.~Zeng, and B.~Xin, ``Distributed proximal-gradient
  algorithms for nonsmooth convex optimization of second-order multiagent
  systems,'' \emph{International Journal of Robust and Nonlinear Control},
  vol.~30, no.~17, pp. 7574--7592, 2020.

\bibitem{disnonsmooth_tac}
S.~A. {Alghunaim}, E.~{Ryu}, K.~{Yuan}, and A.~H. {Sayed}, ``Decentralized
  proximal gradient algorithms with linear convergence rates,'' \emph{IEEE
  Transactions on Automatic Control}, pp. 1--1, 2020.

\bibitem{svm_dis}
A.~Shankar, M.~Vatsa, and P.~B. Sujit, ``A low-cost monocular vision-based
  obstacle avoidance using {SVM} and optical flow,'' \emph{Unmanned Systems},
  vol.~06, no.~04, pp. 267--275, 2018.

\bibitem{zeng_tac}
Y.~{Wei}, H.~{Fang}, X.~{Zeng}, J.~{Chen}, and P.~{Pardalos}, ``A smooth double
  proximal primal-dual algorithm for a class of distributed nonsmooth
  optimization problems,'' \emph{IEEE Transactions on Automatic Control},
  vol.~65, no.~4, pp. 1800--1806, 2020.

\bibitem{nonsmooth_nonconvex_overview}
K.~Khamaru and M.~Wainwright, ``Convergence guarantees for a class of
  non-convex and non-smooth optimization problems,'' in \emph{Proceedings of
  the 35th International Conference on Machine Learning}, vol.~80.\hskip 1em
  plus 0.5em minus 0.4em\relax Stockholm Sweden: PMLR, 10--15 Jul 2018, pp.
  2601--2610.

\bibitem{cen-PALM}
J.~Bolte, S.~Sabach, and M.~Teboulle, ``Proximal alternating linearized
  minimization for nonconvex and nonsmooth problems,'' \emph{Mathematical
  Programming}, vol. 146, pp. 459--494, 08 2013.

\bibitem{Yin_2017}
Y.~{Xu} and W.~{Yin}, ``A globally convergent algorithm for nonconvex
  optimization based on block coordinate update,'' \emph{Journal of Scientific
  Computing}, vol.~72, no.~2, pp. 700--734, 2017.

\bibitem{GuWHH18}
B.~Gu, D.~Wang, Z.~Huo, and H.~Huang, ``Inexact proximal gradient methods for
  non-convex and non-smooth optimization,'' in \emph{Proceedings of the
  Thirty-Second {AAAI} Conference on Artificial Intelligence, (AAAI-18), New
  Orleans, Louisiana, USA}.\hskip 1em plus 0.5em minus 0.4em\relax {AAAI}
  Press, 2018, pp. 3093--3100.

\bibitem{inexact_lei}
J.~Lei and U.~V. Shanbhag, ``Asynchronous schemes for stochastic and
  misspecified potential games and nonconvex optimization,'' \emph{Operations
  Research}, vol.~68, no.~6, pp. 1742--1766, 2020.

\bibitem{sw}
D.~Davis, B.~Edmunds, and M.~Udell, ``The sound of {APALM} clapping: Faster
  nonsmooth nonconvex optimization with stochastic asynchronous {PALM},'' in
  \emph{Advances in Neural Information Processing Systems}, vol.~29.\hskip 1em
  plus 0.5em minus 0.4em\relax Curran Associates, Inc., 2016.

\bibitem{JMLRPGA}
Y.~Zhou, Y.~Liang, Y.~Yu, W.~Dai, and E.~P. Xing, ``Distributed proximal
  gradient algorithm for partially asynchronous computer clusters,''
  \emph{Journal of Machine Learning Research}, vol.~19, no.~19, pp. 1--32,
  2018.

\bibitem{decen-nonconvex}
J.~{Zeng} and W.~{Yin}, ``On nonconvex decentralized gradient descent,''
  \emph{IEEE Transactions on Signal Processing}, vol.~66, no.~11, pp.
  2834--2848, 2018.

\bibitem{hong_nonconv}
H.~{Sun} and M.~{Hong}, ``Distributed non-convex first-order optimization and
  information processing: Lower complexity bounds and rate optimal
  algorithms,'' in \emph{2018 52nd Asilomar Conference on Signals, Systems, and
  Computers}, 2018, pp. 38--42.

\bibitem{scutari_timevarying}
G.~Scutari and Y.~Sun, ``Distributed nonconvex constrained optimization over
  time-varying digraphs,'' \emph{Mathematical Programming}, vol. 176, no.~1,
  pp. 497--544, 2019.

\bibitem{NEXT2016}
P.~D. {Lorenzo} and G.~{Scutari}, ``Next: In-network nonconvex optimization,''
  \emph{IEEE Transactions on Signal and Information Processing over Networks},
  vol.~2, no.~2, pp. 120--136, 2016.

\bibitem{constrianed_prox}
J.~{Lei} and H.~{Chen}, ``Distributed stochastic approximation algorithm with
  expanding truncations,'' \emph{IEEE Transactions on Automatic Control},
  vol.~65, no.~2, pp. 664--679, 2020.

\bibitem{prox_pro}
A.~Beck and M.~Teboulle, ``A fast iterative shrinkage-thresholding algorithm
  for linear inverse problems,'' \emph{SIAM Journal on Imaging Sciences},
  vol.~2, no.~1, pp. 183--202, 2009.

\bibitem{Nash_lou}
Y.~{Lou}, Y.~{Hong}, L.~{Xie}, G.~{Shi}, and K.~H. {Johansson}, ``Nash
  equilibrium computation in subnetwork zero-sum games with switching
  communications,'' \emph{IEEE Transactions on Automatic Control}, vol.~61,
  no.~10, pp. 2920--2935, 2016.

\bibitem{proximal-convex}
A.~I. {Chen} and A.~{Ozdaglar}, ``A fast distributed proximal-gradient
  method,'' in \emph{2012 50th Annual Allerton Conference on Communication,
  Control, and Computing (Allerton)}, Oct 2012, pp. 601--608.

\bibitem{distri-nedich}
A.~{Nedic}, A.~{Ozdaglar}, and P.~A. {Parrilo}, ``Constrained consensus and
  optimization in multi-agent networks,'' \emph{IEEE Transactions on Automatic
  Control}, vol.~55, no.~4, pp. 922--938, 2010.

\bibitem{Schmidt2012}
M.~Schmidt, N.~L. Roux, and F.~Bach, ``Convergence rates of inexact
  proximal-gradient methods for convex optimization,'' \emph{Advances in Neural
  Information Processing Systems}, vol.~24, pp. 1458--1466, 2012.

\bibitem{conver_cauchy}
A.~Stephen, \emph{Understanding Analysis}.\hskip 1em plus 0.5em minus
  0.4em\relax New York, USA: Springer Science Business Media, 2001.

\bibitem{Huang2019}
F.~Huang, B.~Gu, Z.~Huo, S.~Chen, and H.~Huang, ``Faster gradient-free proximal
  stochastic methods for nonconvex nonsmooth optimization,'' \emph{Proceedings
  of the AAAI Conference on Artificial Intelligence}, vol.~33, no.~01, pp.
  1503--1510, Jul. 2019.

\end{thebibliography}

\end{document}